\documentclass[12pt,reqno]{amsart}
\usepackage{amssymb}
\usepackage{graphicx}
\usepackage[all]{xy}
\DeclareFontFamily{U}{mathb}{\hyphenchar\font45}
\DeclareFontShape{U}{mathb}{m}{n}{
      <5> <6> <7> <8> <9> <10> gen * mathb
      <10.95> mathb10 <12> <14.4> <17.28> <20.74> <24.88> mathb12
      }{}
\DeclareSymbolFont{mathb}{U}{mathb}{m}{n}
\DeclareMathSymbol{\righttoleftarrow}{3}{mathb}{"FD}

\theoremstyle{plain}
\newtheorem{prop}{Proposition}[section]
\newtheorem{theo}[prop]{Theorem}

\newtheorem{lemm}[prop]{Lemma}
\theoremstyle{remark}

\newtheorem{assu}{Assumption}
\theoremstyle{definition}
\newtheorem{defi}[prop]{Definition}

\newtheorem{exam}[prop]{Example}
\numberwithin{equation}{section}

\newcommand{\A}{{\mathbb A}}

\newcommand{\bP}{{\mathbb P}}
\newcommand{\Q}{{\mathbb Q}}

\newcommand{\G}{{\mathbb G}}

\newcommand{\Z}{{\mathbb Z}}

\newcommand{\cB}{{\mathcal B}}
\newcommand{\cC}{{\mathcal C}}

\newcommand{\cO}{{\mathcal O}}

\newcommand{\cL}{{\mathcal L}}
\newcommand{\cX}{{\mathcal X}}
\newcommand{\cY}{{\mathcal Y}}
\newcommand{\cZ}{{\mathcal Z}}

\newcommand{\cV}{{\mathcal V}}

\newcommand{\rK}{{\mathrm K}}
\newcommand{\upperhalfplane}{{\mathbb H}}
\newcommand{\eqto}{\stackrel{\lower1.5pt\hbox{$\scriptstyle\sim\,$}}\to}
\newcommand{\eqdashto}{\stackrel{\lower1.5pt\hbox{$\scriptstyle\sim\,$}}\dashrightarrow}
\newcommand{\actsfromleft}{\mathrel{\reflectbox{$\righttoleftarrow$}}}
\newcommand{\actsfromright}{\righttoleftarrow}

\DeclareMathOperator{\Pic}{Pic}
\DeclareMathOperator{\Spec}{Spec}

\DeclareMathOperator{\Var}{Var}
\DeclareMathOperator{\Burn}{Burn}

\newcommand{\ocB}{{\overline\cB}}
\newcommand{\oBurn}{{\overline\Burn}}

\begin{document}
\title[Birational types of orbifolds]{Birational types of algebraic orbifolds}
\author{Andrew Kresch}
\address{
  Institut f\"ur Mathematik,
  Universit\"at Z\"urich,
  Winterthurerstrasse 190,
  CH-8057 Z\"urich, Switzerland
}
\email{andrew.kresch@math.uzh.ch}
\author{Yuri Tschinkel}
\address{
  Courant Institute,
  251 Mercer Street,
  New York, NY 10012, USA
}

\email{tschinkel@cims.nyu.edu}

\address{Simons Foundation\\
160 Fifth Avenue\\
New York, NY 10010\\
USA}

\date{November 4, 2019}

\begin{abstract}
We introduce a variant of the birational symbols group of
Kontsevich, Pestun, and the second author, and use this to define
birational invariants of algebraic orbifolds.
\end{abstract}

\maketitle

\section{Introduction}
\label{sec.introduction}
Let $k$ be a field of characteristic zero and $X$ a smooth projective
variety over $k$, of dimension $n$;
we require our varieties to be irreducible, but not necessarily
geometrically irreducible.
The paper \cite{kontsevichtschinkel} introduced the
\emph{Burnside group} of varieties
$$
\Burn_n=\Burn_{n,k},
$$
the free abelian group on isomorphism classes of
finitely generated fields of transcendence degree $n$ over $k$; for such a field
$K$ we denote the corresponding generator by $[K]$.
To $X$ one associates its class
$$
[X]:=[k(X)]\in \Burn_n,
$$
extended by additivity for general
smooth projective schemes.
To
$$
U\subset X \setminus D,
$$
the complement to a simple normal crossing
divisor
$$
D=D_1\cup\dots\cup D_\ell,
$$
one may also associate a
class in $\Burn_n$:
\begin{equation}
\label{eqn.U}
[U]:=[X]-\sum_{1\le i\le \ell} [D_i\times \bP^1]+
\sum_{1\le i<j\le \ell} [(D_i\cap D_j)\times \bP^2]-\dots.
\end{equation}
This is not only an invariant of the isomorphism type of $U$,
but is a birational invariant in the following sense:
$[U]=[U']$ in $\Burn_n$ if there exist a quasiprojective variety $V$ and
birational projective morphisms
\[ V\to U\qquad\text{and}\qquad V\to U'. \]
This formalism was used to establish specialization of rationality.

Now we suppose that $X$ is equipped with a faithful action of a
finite abelian group $G$.
Then $\Burn_n$ is replaced by the equivariant Burnside group
$$
\Burn_n(G),
$$
introduced in \cite{kontsevichpestuntschinkel}
(in a slightly different form, as explained in Appendix \ref{app.BnG}).
This is the quotient by suitable relations of
the free abelian group on
triples consisting of:
\begin{itemize}
\item
a subgroup $H\subset G$,
\item a $G/H$-Galois algebra
extension $K$ of a field $K_0$ of transcendence degree $d\le n$ over $k$,
up to equivariant isomorphism,
required to satisfy
Assumption \ref{assu.Hexponent} in Appendix \ref{app.BnG}
(a technical condition, always satisfied when
$k$ is algebraically closed),
and
\item
a faithful $(n-d)$-dimensional linear representation of $H$ over $K_0$
with trivial space of invariants, up to equivalence.
\end{itemize}
Then an invariant of the $G$-equivariant birational type of $X$ is obtained
from a stratification of $X$ by the stabilizer group $H\subset G$ as the sum
\begin{equation}
\label{eqn.GX}
\sum_{H\subset G}
\sum_{\substack{Y\subset X\\ \text{with stabilizer $H$}}}
[(G/H \actsfromleft k(Y),\beta_Y(X)],
\end{equation}
where the inner sum is over unions $Y$ of $G$-orbits of components,
with $k(Y)$ understood as the product
of function fields of the components of $Y$,
and where $\beta_Y(X)$ records the generic normal bundle representation along $Y$.
There is an analogous $G$-equivariant birational invariant of $U$ as above,
where each $D_i$ is assumed to be $G$-invariant.

This paper concerns birational invariants of
(quasi)projective {\em orbifolds} $\cX$.
Here, by an orbifold we mean a smooth separated irreducible
finite-type Deligne-Mumford stack
over $k$ with trivial generic stabilizer.
Such a stack has a coarse moduli space \cite{keelmori}, a
separated algebraic space of finite type over $k$.
Following \cite{kreschseattle}, we say that the orbifold $\cX$ is
\emph{quasiprojective} when the
coarse moduli space is a quasiprojective variety, and is
\emph{projective} when the
coarse moduli space is a projective variety.
By Theorems 4.4 and 5.3 of op.\ cit.,
every quasiprojective orbifold may be presented as a locally closed substack
of a projective orbifold.

We will introduce a variant
$$
\oBurn_n
$$
of the groups $\Burn_n$ and  $\Burn_n(G)$.
In essence, we only carry in $\oBurn_n$ the information of
representations of finite abelian groups, up to automorphisms of those groups.
Working with $\oBurn_n$, we exhibit a
birational invariant of a quasiprojective $n$-dimensional orbifold $\cX$.

It suffices to consider finite {\em abelian} groups thanks to the
divisorialification procedure of \cite{berghrydh}, a sequence of blow-ups in
smooth centers which,
when applied to a general orbifold, yields an orbifold with only
abelian groups as geometric stabilizer groups.
Weak factorization \cite{AKMW}, in a
functorial form proved in \cite{abramovichtemkin},
is used to exhibit the desired birational invariance.

In Section \ref{sec.scissors} we recall the Burnside group of varieties
and establish a presentation by scissors-like relations, analogous to
the scissors relations defining the Grothendieck group of varieties.
Section \ref{sec.hebg} introduces the orbifold version $\oBurn_n$,
where the birational invariant of quasiprojective orbifolds
defined in Section \ref{sec.birat} takes its value
(Theorem \ref{thm.projective}).
A computation of invariants of orbifold surfaces reveals an
intriguing connection with modular curves (Proposition \ref{prop.X0p}),
whose proof is given in Section \ref{sec.manin}.

\medskip
\noindent
\textbf{Acknowledgments:}
The first author was partially supported by the
Swiss National Science Foundation.

\section{Burnside group via scissors relations}
\label{sec.scissors}
Let $k$ be a field of characteristic zero.
The Grothendieck group
$$
\rK_0(\Var_k)
$$
may be approached in two ways,
as an abelian group generated by the classes of algebraic varieties over $k$
with the classical scissors relations
(where it makes no difference if we restrict to just
smooth quasiprojective varieties), or via the Bittner presentation
\cite{bittner}, which only involves smooth projective varieties.
Of course, $\rK_0(\Var_k)$ has a ring structure as well,
but we do not concern ourselves with this here.

In this section we record the
observation that the Burnside group $\Burn_n$ also admits a
description in terms of scissors relations.
As mentioned in the Introduction, we only require our varieties to be
irreducible (but not necessarily geometrically irreducible).

\begin{lemm}
\label{lem.divisors}
Let $k$ be a field of characteristic zero, and let
$W$ be a smooth quasiprojective variety over $k$.
For any nonempty open $U\subset W$ there exist
divisors $D_1$, $\dots$, $D_\ell$ such that $W\smallsetminus D_1$ is
contained in $U$,
and $D_1\smallsetminus D_2$, $\dots$, $D_{\ell-1}\smallsetminus D_\ell$,
$D_\ell$ are all smooth.
\end{lemm}

\begin{proof}
Let $Z=W\smallsetminus U$.
By \cite[Thm.\ 7]{kleimanaltman},
given an embedding of $W$ in projective space,
a general hypersurface of
sufficiently large degree containing $Z$ defines a divisor
$D_1$ on $W$ whose singular locus $D_1^{\mathrm{sing}}$ is
contained in $Z$ and does not contain any irreducible component of $Z$.
If $D_1$ is smooth, then we are done with $\ell=1$.
Otherwise, we have $\dim(D_1^{\mathrm{sing}})<\dim(Z)$, and we conclude
by induction on $\dim(Z)$.
\end{proof}

\begin{prop}
\label{prop.scissors}
Let $k$ be a field of characteristic zero and $n$ a natural number.
Then the assignment to $[k(X)]$ of $[X]$ for smooth projective varieties $X$ of
dimension $n$ over $k$ defines an isomorphism
\[
\Burn_n\stackrel{\lower1.5pt\hbox{$\scriptstyle\sim\,$}}{\longrightarrow} \Big(\bigoplus_{[U],\,\dim(U)=n} \Z\cdot [U]\Big)\Big/\text{modified-scissors},
\]
where, on the right, we have the
quotient of the free abelian group on isomorphism classes of
smooth quasiprojective varieties of dimension $n$ over $k$
by the \emph{modified scissors relations}
\[ [U]=[V\times \bP^{n-d}]+[U\smallsetminus V] \]
for smooth closed subvarieties $V\subset U$ of dimension $d<n$.
The inverse isomorphism is given by the formula \eqref{eqn.U}.
\end{prop}

\begin{proof}
We check that the map from the statement of the proposition is
well-defined, i.e., the classes of any pair of birationally equivalent
smooth projective $n$-dimensional varieties are equal modulo
the modified scissors relations.

By weak factorization, it suffices to consider the case of
$X$ and $B\ell_YX$, where $X$ is smooth and projective of dimension $n$
and $Y$ is a smooth subvariety of $X$ of dimension $d<n$.
We have
\begin{align*}
[X]&=[Y\times \bP^{n-d}]+[X\smallsetminus Y],\\
[B\ell_YX]&=[\bP(N_{Y/X})\times \bP^1]+[X\smallsetminus Y],
\end{align*}
where $N_{Y/X}$ denotes the normal bundle.
We are done if we can show that
$[\bP(N_{Y/X})\times \bP^1]=[Y\times \bP^{n-d}]$.
We will show, more generally,
that for any smooth quasiprojective variety $W$ of dimension $e<n$
and vector bundle $F$ on $W$ of rank $r\le n+1-e$, we have
\begin{equation}
\label{eqn.PF}
[\bP(F)\times \bP^{n+1-e-r}]=[W\times \bP^{n-e}].
\end{equation}
For any smooth quasiprojective variety $Z$ of
dimension $n-1$ we have $[Z\times \A^1]=0$
(by considering $Z\times \{\infty\}\subset Z\times \bP^1$),
and hence
$$
[W\times \bP^{n-e}]=[W\times (\bP^1)^{n-e}]
$$
(by considering $W\times \bP^{n-e-1}\subset W\times \bP^{n-e}$).
We prove \eqref{eqn.PF}
by induction on $e$; the case $e=0$ is now clear.
Let $U\subset W$ be a nonempty open subset on which $F$ is trivial, and
$D_1$, $\dots$, $D_\ell$, divisors as in
Lemma \ref{lem.divisors}.
The modified scissors relation and the induction hypothesis lead to
\begin{align*}
[\bP(F)\times \bP^{n+1-e-r}] &=
[D_\ell\times \bP^{n+1-e} ]+
[(D_{\ell-1}\smallsetminus D_\ell)\times \bP^{n+1-e} ] \\
&\,\,\,\,\,\,\,\,+\dots+[(D_1\smallsetminus (D_2\cup\dots\cup D_\ell)\times \bP^{n+1-e} ] \\
&\,\,\,\,\,\,\,\,\,\,\,\,\,\,\,\,\,\,\,\,\,\,\,\,\,\,+ [(W\smallsetminus(D_1\cup\dots\cup D_\ell))\times \bP^{n-e}].
\end{align*}
We conclude with the relations, for $1\le i\le \ell$:
\begin{align*}
[(W\smallsetminus (D_{i+1}\cup\dots\cup D_\ell))\times \bP^{n-e}]&=
[(D_i\smallsetminus (D_{i+1}\cup\dots\cup D_\ell))\times \bP^{n+1-e}] \\
&\,\,\,\,\,\,\,\,+
[(W\smallsetminus (D_i\cup\dots\cup D_\ell))\times \bP^{n-e}].
\end{align*}

Now we verify that the map in the reverse direction,
given by the formula \eqref{eqn.U}, is well-defined,
i.e., respects the modified scissors relations.
Let $V$ be a smooth closed subvariety of $U$ of dimension $d$.
Then $U$ may be presented as the complement in a smooth projective variety $X$
of a simple normal crossing divisor
$D_1\cup\dots\cup D_\ell$, with which a smooth subvariety $Y\subset X$ has
normal crossing, such that $Y\cap U=V$.
We have $[U]$, given by the formula \eqref{eqn.U}.
For $[V\times \bP^{n-d}]$ we have the embedding in $Y\times \bP^{n-d}$,
complement to the simple normal crossing divisor
$$
(D_1\cap Y)\times \bP^{n-d}\cup\dots\cup (D_\ell\cap Y)\times \bP^{n-d},
$$
and thus an analogous formula in $\Burn_n$.
The blow-up $B\ell_YX$ has the simple normal crossing divisor
$\widetilde{D}_1\cup\dots\cup \widetilde{D}_\ell\cup E$,
where $\widetilde{D}_i$ denotes the proper transform of $D_i$,
and $E$, the exceptional divisor, leading to
a formula for $[U\smallsetminus V]$ in $\Burn_n$.
Comparing formulas and using that any intersection not involving $E$
is birational to the corresponding intersection in $X$, while any
intersection involving $E$ is birational to the product of
an intersection in $Y$ with projective space of the appropriate dimension,
we obtain the desired relation.

That the composite of the forward and reverse maps, in either order, is
the identity, is clear for the composite $\Burn_n\to \Burn_n$, and for
the other, comes down to iterated application of the modified scissors relations.
\end{proof}

\section{Burnside group for stacks}
\label{sec.hebg}
In this section we introduce a variant of the equivariant Burnside group
which is adapted to the setting of orbifolds.

\begin{defi}
\label{def.ocB}
We define the $\Z[t]$-module $\ocB$ by starting with
the free $\Z$-module on pairs $(A,S)$ consisting of a
finite abelian group $A$ and finite generating system $S$ of $A$,
where the action of $t$ is to append the element $0$ to $S$,
and passing to the quotient by the following relations:
\begin{itemize}
\item $(A,S)$ and $(A,S')$ are equivalent if $S'$ is a permutation of $S$.
\item $(A,S)$ and $(A',S')$ are equivalent if some
isomorphism $A\cong A'$ transforms $S$ to $S'$.
\item $(A,S)$, $S=(a_1,\dots,a_m)$, is equivalent, for any $2\le j\le m$, to
\begin{align*}
\sum_{\emptyset\ne I\subset \{1,\dots,j\}} &
(-t)^{|I|-1}
\Big(A/\langle a_i{-}a_{i_0}\rangle_{i\in I},\\
&(\bar a_{i_0},\bar a_1{-}\bar a_{i_0},\dots\text{(omitting all $i\in I$)}\dots,
\bar a_j{-}\bar a_{i_0},\bar a_{j+1},\dots,\bar a_m)\Big),
\end{align*}
where inside the sum $i_0$ denotes an element of $I$,
with sequence of elements of
$A/\langle a_i{-}a_{i_0}\rangle_{i\in I}$ of length $1+(j-|I|)+(m-j)$
that is independent of the choice of $i_0$.
\end{itemize}
\end{defi}

We let $[A,S]$ denote the class in $\ocB$ of a pair $(A,S)$.
The natural grading on $\Z[t]$ yields
a grading on $\ocB$ that assigns
degree $|S|$ to $[A,S]$:
\[ \ocB=\bigoplus_{n=0}^\infty \ocB_n. \]
Representations determine elements of $\ocB$:
if $G$ is a finite diagonalizable group scheme with
representation
$$
\rho\colon G\to GL_n
$$
(over an arbitrary field), then there is
an associated element
\[ [\rho]\in \ocB_n, \]
given by the Cartier dual group, with the
sequence of weights supplied by a decomposition of $\rho$ as a sum of $n$
one-dimensional linear representations.

Restricting to $e$-torsion groups $A$ for a positive integer $e$, respectively,
to $p$-primary $A$ for a prime number $p$, leads to a $\Z[t]$-module
$\ocB^{[e]}$, respectively $\ocB^{(p)}$.
The evident homomorphisms from these modules to $\ocB$ are
split monomorphisms, with splittings given by
$$
[A,S]\to [A/eA,S],\qquad\text{respectively,}\qquad [A,S]\to [A(p),S],
$$
where $A(p)$ denotes the
$p$-primary subgroup of $A$.
We have
\[ \ocB=\bigoplus_p \ocB^{(p)},\qquad
\ocB^{(p)}=\varinjlim_j \ocB^{[p^j]}. \]

\begin{defi}
\label{def.oBurn}
Let $k$ be a field of characteristic zero
and $n$ a natural number.
The group
$$
\oBurn_n
$$
is the
abelian group generated by pairs $(K,\alpha)$, where
\begin{itemize}
\item
$K$ is a field
of transcendence degree $d\le n$ over $k$ and
\item
$\alpha\in \ocB_{n-d}$,
\end{itemize}
modulo the identification of $(K(t),\beta)$ and $(K,t\beta)$ for
$\beta\in \ocB_{n-d-1}$.
\end{defi}

\begin{exam}
\label{exa.oBurn}
For $\ocB_2^{[5]}$ we have generators
$t^2[0,()]$,
$t[C_5,(1)]$,
$[C_5,(1,1)]$,
$[C_5,(1,2)]$,
$[C_5,(1,4)]$,
$[C_5\oplus C_5,((1,0),(0,1))]$, and relations:
\begin{align*}
t[C_5,(1)]&=[C_5,(1,4)]+t[C_5,(1)]-t^2[0,()], \\
[C_5,(1,1)]&=2t[C_5,(1)]-t[C_5,(1)], \\
[C_5,(1,2)]&=[C_5,(1,1)]+[C_5,(1,2)]-t^2[0,()], \\
[C_5,(1,4)]&=2[C_5,(1,2)]-t^2[0,()], \\
[C_5\oplus C_5,((1,0),(0,1))]&=
2[C_5\oplus C_5,((1,0),(0,1))]-t[C_5,(1)],
\end{align*}
where $C_5=\Z/5\Z$.
We deduce
\[ [C_5\oplus C_5,((1,0),(0,1))]=[C_5,(1,1)]=[C_5,(1,4)]
=t[C_5,(1)]=t^2[0,()], \]
with
\[ 2\big([C_5,(1,2)]-t^2[0,()]\big)=0. \]
As may be seen directly, or by application of Theorem \ref{thm.projective},
below, over an algebraically closed field of characteristic zero,
among rational orbifold surfaces whose only nontrivial stabilizer groups have
order $5$, the parity of the number of isolated points with $C_5$-stabilizer and
unequal weights not summing to zero remains invariant under blow-ups of points.
As noted in \cite[Exa.\ 4.3]{bergh}, it is not possible to eliminate such an
isolated point with $C_5$-stabilizer just with blow-ups of points.
\end{exam}

\section{Birational invariants of orbifolds}
\label{sec.birat}
In this section we introduce new birational invariants of $n$-dimensional
orbifolds over a field $k$ of characteristic zero,
taking values in $\oBurn_n$.

Let $\cX$ be an orbifold.
We recall from \cite{bergh} (see also \cite{berghrydh}):
if $D_1\cup\dots\cup D_\ell$ is a
simple normal crossing divisor on $\cX$, then
$\cX$ is called \emph{divisorial} with respect to
$D_1$, $\dots$, $D_\ell$
if the morphism
$$
\cX\to B\G_m^\ell,
$$
determined by
$\cO_{\cX}(D_i)$, for $i=1$, $\dots$, $\ell$, is representable.
We will apply this terminology more generally to any
a finite collection of line bundles.

Divisorialification is a procedure
that, when applied to an orbifold $\cX$,
yields a succession of blow-ups along smooth centers
\[ \cY\to\dots\to \cX, \]
such that $\cY$ is divisorial with respect to a
suitable simple normal crossing divisor.
This is given as Algorithm C in \cite{bergh},
initially with a requirement to have abelian geometric stabilizer groups,
later with this requirement removed \cite{berghrydh}.

As explained in the introduction,
invariance under birational projecive morphisms is the statement of
invariance under the equivalence relation of existence of a third
object (variety or Deligne-Mumford stack) with birational projective morphisms
to two given objects.
In this section we are interested in
quasiprojective orbifolds $\cX$ and $\cX'$, and the equivalence takes the form
of existence of a Deligne-Mumford stack $\cY$ with birational projective morphisms
\[ \cY\to \cX\qquad\text{and}\qquad \cY\to \cX'. \]
There is no loss of generality in supposing $\cY$ as well to be an orbifold,
since resolution of singularities in a
functorial form as in \cite{villamayor} and \cite{bierstonemilman} is
applicable to algebraic stacks.
Here we remind the reader that when $\cX$ and $\cY$ are
quasiprojective orbifolds, a representable morphism $\cY\to \cX$
is projective if and only if it is proper.
(Every projective morphism is proper.
The reverse implication uses that
$\cY\to \cX$ factors up to  a $2$-isomorphism
through $\cX\times_XY$, where $X$ and $Y$ denote the respective
coarse moduli spaces, that
$\cX\to X$ and $\cY\to Y$ induce bijections on geometric points,
and that a representable proper morphism inducing a bijection on
geometric points is finite, hence projective.)

\begin{theo}
\label{thm.projective}
Let $k$ be a field of characteristic zero,  $n$ a natural number, and
$\cX$  an $n$-dimensional quasiprojective orbifold over $k$.
The following recipe, assigning to $\cX$ a class $[\cX]\in \oBurn_n$
gives an invariant under representable birational projective morphisms:
\begin{itemize}
\item Use divisorialification to replace $\cX$ by a quasiprojective orbifold
$\cY$ that is divisorial
with respect to some finite collection of line bundles.
\item Stratify $\cY$ by the isomorphism type of the geometric stabilizer group
and attach to each component the normal bundle:
\[ \cY=\coprod_G \cY_G,\qquad N_{Y,G}=N_{\cY_G/\cY}. \]
\item Writing the coarse moduli space of $\cY_G$, for each $G$, as
$Y_G$, we assign the element
\[ [\cX]:=\sum_G ([Y_G],[N_{Y,G}])\in \oBurn_n. \]
\end{itemize}
In the last step,
if $Y_G$ is irreducible of dimension $d$, then we understand
$[Y_G]$ to be the associated element of $\Burn_d$, with
$[N_{Y,G}]\in \ocB_{n-d}$ associated to the representation of $G$ at the
geometric generic point of $\cY_G$.
In general, we understand $([Y_G],[N_{Y,G}])$ to be the sum of
the elements of $\oBurn_n$ attached to the irreducible components.
\end{theo}

\begin{proof}
Let $\cX'$ be a quasiprojective orbifold with
representable birational projective morphism to $\cX$.
We divisorialize $\cX'$ to obtain $\cY'$.
The diagram
\[
\xymatrix{
& \cY' \ar[d] \\
\cY \ar[r] & \cX
}
\]
may be completed to a $2$-commutative square of
representable birational projective morphisms of
quasiprojective orbifolds by desingularizing the
closure in the fiber product
of a nonempty open substack where the morphisms are isomorphisms.
This way, we are reduced to showing that for a representable birational projective morphism
$\cZ\to \cY$
of quasiprojective orbifolds we have
\begin{equation}
\label{eqn.YeqZ}
\sum_G ([Y_G],[N_{Y,G}])=\sum_G ([Z_G],[N_{Z,G}]) \in \oBurn_n.
\end{equation}

Let $\cL_1$, $\dots$, $\cL_\ell$ be line bundles, relative to which
$\cY$ is divisorial.
The functorial form of weak factorization
in \cite{abramovichtemkin} is applicable to stacks and
yields a factorization of
$\cZ\to \cY$ as a composite of maps
\emph{of divisorial projective orbifolds}
(with respect to pullbacks of
$\cL_1$, $\dots$, $\cL_\ell$), each equal to or inverse to
a blow-up along a smooth center.

Let $\cV$ be a smooth closed substack of $\cY$ of dimension $d<n$, with
coarse moduli space $V$,
and let $\cZ=B\ell_{\cV}\cY$.
We verify \eqref{eqn.YeqZ} in this case.
On the left, we break up $\cY_G$ into the unions of components
$\cY'_G$ disjoint from $\cV$ and
$\cY''$, meeting $\cV$ nontrivially, and
apply the modified scissors relation to $\cY''_G$:
\begin{align*}
&\sum_G ([Y_G],[N_{Y,G}])=
\sum_G ([Y'_G],[N_{Y,G}]) \\
&\,\,+\sum_G
([Y''_G\cap V,t^{\dim(\cY''_G)-\dim(\cY''_G \cap \cV)}[N_{Y,G}])
+\sum_G ([Y''_G\smallsetminus V],[N_{Y,G}]),
\end{align*}
where in the second sum on the right, the dimensions are understood to be
taken componentwise.
Breaking up the sum on the right of \eqref{eqn.YeqZ} in a similar fashion,
we obtain an expression with identical first and third sums and
a second sum that differs from the second sum in the expression above
by relations in $\ocB$.
\end{proof}

\begin{exam}
\label{exa.stackiness}
Functorial destackification \cite{bergh} of an orbifold provides a sequence of
blow-ups along smooth centers and root stack operations along smooth divisors
that simplify the stack structure.
The root stack operation adds stabilizer $\mu_n$ (for some positive integer $n$)
along a divisor \cite[\S 2]{cadman}, \cite[App.\ B]{AGV},
and the outcome of destackification is an orbifold that
is obtained from a smooth variety by iterating root stack operations along
components of a simple normal crossing divisor.
Blow-ups alone are, as noted in Example \ref{exa.oBurn},
insufficient to bring a general orbifold into this form.
Correspondingly, we view the quotient
$
\ocB/\cC,
$
where $\cC$ denotes the
submodule generated by the classes of pairs
\[ (C_{a_1}\oplus\dots\oplus C_{a_r},(g_1,\dots,g_r)) \]
of direct sums of finite cyclic groups and
tuples of generators, as an invariant of an orbifold
up to smooth blow-ups and root stacks.
We have
$$
\ocB^{[p]}\subset \cC\qquad\text{for}\qquad p\in \{2,3\},
$$
since
blow-ups suffice for the destackification in these cases
\cite{kreschdestackification}, \cite{oesinghaus}.

Table \ref{table1}, which records the
isomorphism type of
$\ocB_2^{[p]}/(\cC\cap \ocB_2^{[p]})$ for some primes $p\ge 5$,
reveals a pattern that we are able to confirm.
\begin{table}
\[
\begin{array}{rccrccrc}
p&\ocB_2^{[p]}/(\cC\cap \ocB_2^{[p]})&&
p&\ocB_2^{[p]}/(\cC\cap \ocB_2^{[p]})&&
p&\ocB_2^{[p]}/(\cC\cap \ocB_2^{[p]}) \\ \hline
5&\Z/2\Z&&17&\Z/2\Z\oplus\Z&&31&\Z^2 \\
7&0&&19&\Z&&37&\Z/2\Z\oplus\Z^2 \\
11&\Z&&23&\Z^2&&41&\Z/2\Z\oplus\Z^3 \\
13&\Z/2\Z&&29&\Z/2\Z\oplus\Z^2&&43&\Z^3
\end{array}
\]
\caption{Isomorphism type of $\ocB_2^{[p]}/(\cC\cap \ocB_2^{[p]})$}
\label{table1}
\end{table}

\begin{prop}
\label{prop.X0p}
For a prime $p\ge 5$ let 
$$
g=g(X_0(p))
$$ 
denote the genus of the modular curve, i.e., 
$$
g= \begin{cases} \big[\frac{p}{12}\big]\mp 1, & \text{when $p\equiv \pm 1$ $\mathrm{mod}$ $12$}, \\
\big[\frac{p}{12}\big], & \text{otherwise}.
\end{cases}
$$
Then
\[
\ocB_2^{[p]}/(\cC\cap \ocB_2^{[p]})\cong
\begin{cases}
\Z/2\Z\oplus \Z^g, & \text{if $p\equiv 1$ $\mathrm{mod}$ $4$}, \\
\Z^g, & \text{if $p\equiv 3$ $\mathrm{mod}$ $4$}.
\end{cases}
\]
\end{prop}

The proof of Proposition \ref{prop.X0p}, based on computations with 
Manin's modular symbols \cite{maninparabolic}, is given in
the next section.

The entry $0$ in Table \ref{table1} for $p=7$ indicates that
$\ocB_2^{[7]}\subset \cC$.
In fact, we have
$\ocB_3^{[7]}\subset \cC$ as well.
But we find
$$
\ocB_4^{[7]}/(\cC\cap \ocB_4^{[7]})\cong \Z/2\Z.
$$
\end{exam}

\section{Modular symbols and the proof of Proposition \ref{prop.X0p}}
\label{sec.manin}
The equivariant Burnside group introduced in
\cite{kontsevichpestuntschinkel} is shown to exhibit
a tantalizing connection with the modular curves $X_1(N)$ for
various $N$.
Here, we see the appearance of the modular curves
\[ X_0(p)=\Gamma_0(p)\backslash\upperhalfplane\cup \{\mathrm{0,\infty}\} \]
and the corresponding modular symbols \cite{maninparabolic}. 

Fix a prime $p\ge 5$; we are interested in the isomorphism type
of the abelian group
$$
\ocB_2^{[p]}/(\cC\cap \ocB_2^{[p]})
$$
with generators
\[ [C_p,(1,a)],\qquad \,2\le a\le p-2, \]
and relations
\begin{align}
[C_p,(1,a)]&=[C_p,(1,a^{-1})]\qquad\text{for all $a$}, \label{rel1} \\
2[C_p,(1,2)]&=0, \label{rel2} \\
[C_p,(1,2)]&=-[C_p,(1,p-2)], \label{rel3} \\
[C_p,(1,a)]&=[C_p,(1,a-1)]+[C_p,(1,a^{-1}-1)] \label{rel4} \\
&\textstyle\qquad\qquad\text{for $a\in
\big\{3,\dots,\frac{p-1}{2}\big\}\cup\big\{\frac{p+3}{2},\dots,p-2\big\}$},
\nonumber
\end{align}
where $a^{-1}$ denotes the positive integer less than $p$,
inverse to $a$ mod $p$.

The modular group
\[ \Gamma_0(p)=\bigg\{\begin{pmatrix}a&b\\c&d\end{pmatrix}\in \mathrm{SL}_2(\Z)\,\bigg|\,
c\equiv 0 \text{ mod } p \bigg\} \]
has index $p+1$ in $\mathrm{SL}_2(\Z)$, with right coset representatives
\[ \begin{pmatrix} 1&0\\0&1 \end{pmatrix},
\begin{pmatrix} 1&0\\1&1 \end{pmatrix},
\dots,
\begin{pmatrix} 1&0\\p-1&1 \end{pmatrix},
\begin{pmatrix} 0&-1\\1&0 \end{pmatrix}. \]
We let $\Gamma_0(p)$
act in the standard way on the upper half-plane $\upperhalfplane$ and as well on
$\Q\cup \{i\infty\}$, the latter with two orbits corresponding to
the cusps $0$, $\infty\in X_0(p)$.
Here, $0$ corresponds to
the set of all $b/d\in \Q$ with $p\nmid d$ and $\infty$, to
the set of $a/c\in \Q$ with $p\mid c$.
The real structure on $X_0(p)$ is determined by the standard complex conjugation
$\upperhalfplane\to \upperhalfplane$, $z\mapsto -\bar z$.
It is well known that the real locus of $X_0(p)$ is connected.

With Manin's modular symbols \cite{maninparabolic},
applied to $\Gamma_0(p)$, we get a presentation of
$H_1(X_0(p),\Z)$ by generators and relations.
Proposition \ref{prop.X0p} is established by showing that these
relations, together with the additional relations that the sum of any
cycle and its complex conjugate is zero, match the presentation
\eqref{rel1}--\eqref{rel4}.
In fact, we use a simpler set of
relations, which yield the homology not of the Riemann surface $X_0(p)$, but
rather of the corresponding \emph{orbifold curve} which carries
orbifold structure at elliptic points.
The quotient of $\upperhalfplane$ by $\Gamma_0(p)/\{\pm 1\}$ is an orbifold,
compactified by adding the cusps to obtain the orbifold curve
\[ X_0(p)_{\mathrm{orb}}. \]

Orbifolds and their topological invariants are explained, for instance, in
\cite{moerdijk}, while
a convenient reference for orbifold curves is \cite{behrendnoohi}.
However,
$H_1(X_0(p)_{\mathrm{orb}},\Z)$
may also be presented directly
as the homology of the complement of the
elliptic points, modulo the relation that an appropriate multiple of a
small loop around an elliptic point is zero.
When $p\equiv 1$ mod $4$ there is a complex conjugate pair of
elliptic points of $H_1(X_0(p)_{\mathrm{orb}},\Z)$ where the
stabilizer (of a representative point of $\upperhalfplane$) has
order $2$ in $\Gamma_0(p)/\{\pm 1\}$; for each of these, twice a small loop
is declared to be zero in homology.
When $p\equiv 1$ mod $6$ there is a complex conjugate pair of
elliptic points where the
stabilizer has order $3$ in $\Gamma_0(p)/\{\pm 1\}$, for which we declare
$3$ times a small loop to be zero in homology.

We summarize the needed results from \cite{maninparabolic},
modified appropriately to the orbifold setting.
We maintain the convention from \eqref{rel1}--\eqref{rel4} about
$a$ and $a^{-1}$ and, when $a\notin \{p-2,(p-1)/2\}$ define
positive integers $a'$ and $a''$ less than $p$ by the requirements
\[ a'\equiv -a^{-1}-1\text{ mod }p,\qquad
a''\equiv -(a+1)^{-1}\text{ mod }p. \]

\begin{lemm}[{\cite[(1.4)]{maninparabolic}}]
\label{lem.toH1}
A surjective homomorphipsm
\[ \Gamma_0(p)\to H_1(X_0(p)_{\mathrm{orb}},\Z) \]
is defined by sending $\gamma\in \Gamma_0(p)$ to the image
\[ \{0,\gamma\cdot 0\} \]
in $X_0(p)$
of a geodesic path in $\upperhalfplane\cup \Q$ from
$0$ to $\gamma\cdot 0$.
The kernel is generated by the commutator subgroup of
$\Gamma_0(p)$ and the parabolic elements of $\Gamma_0(p)$.
\end{lemm}

\begin{lemm}[{\cite[(1.5)--(1.9)]{maninparabolic}}]
\label{lem.presH1}
The abelian group $H_1(X_0(p)_{\mathrm{orb}},\Z)$ is presented by generators
\[ \big\{0,\frac{1}{a}\big\},\qquad 2\le a\le p-2, \]
and relations
\begin{align}
\big\{0,\frac{1}{a}\big\}+\big\{0,\frac{1}{p-a^{-1}}\big\}&=0, \label{mrel1} \\
\big\{0,\frac{1}{a}\big\}+\big\{0,\frac{1}{a'}\big\}+\big\{0,\frac{1}{a''}\big\}
&=0, \label{mrel2} \\
\big\{0,\frac{1}{(p-1)/2}\big\}+\big\{0,\frac{1}{p-2}\big\}&=0.
\label{mrel3}
\end{align}
\end{lemm}

Now the proof of Proposition \ref{prop.X0p} combines an algebraic result
with topological reasoning.

\begin{lemm}
\label{lem.algebraic}
An isomorphism
\begin{align*}
\ocB_2^{[p]}/(&\cC\cap \ocB_2^{[p]})\to \\
&H_1(X_0(p)_{\mathrm{orb}},\Z)
\Big/\Big\langle
\big\{0,\frac{1}{a}\big\}+\big\{0,\frac{1}{p-a}\big\},\,a\in
\{2,\dots,p-2\}\big\}\Big\rangle
\end{align*}
is given by $[C_p,(1,a)]\mapsto \{0,1/a\}$ for all $a$.
\end{lemm}

\begin{proof}
Suppose $2\le b\le (p-3)/2$.
We subtract the relations \eqref{rel4} corresponding to $a=b+1$ and $a=p-b$,
noticing that the rightmost terms cancel thanks to \eqref{rel1},
to obtain
\[ [C_p,(1,b+1)]-[C_p,(1,p-b)]=[C_p,(1,b)]-[C_p,(1,p-b-1)]. \]
Starting from \eqref{rel3} we obtain, inductively,
\begin{equation}
\label{rel5}
[C_p,(1,a)]=-[C_p,(1,p-a)]
\end{equation}
for all $a$.
Using \eqref{rel5} and \eqref{rel1}, we rewrite \eqref{rel4} as
\begin{equation}
\label{rel6}
[C_p,(1,a)]+[C_p,(1,a')]+[C_p,(1,a'')]=0
\end{equation}
for $a\notin \{(p-1)/2,p-2\}$.
We conclude by matching relations \eqref{rel1}--\eqref{rel2},
\eqref{rel5}--\eqref{rel6}
with \eqref{mrel1}--\eqref{mrel3} and the
additional relations from the quotient group in the
statement of the lemma.
\end{proof}

While $H_1(X_0(p),\Z)$ is free of rank $2g$
(where $g$ is the genus of $X_0(p)$), there may be torsion in
$H_1(X_0(p)_{\mathrm{orb}},\Z)$:
\[ H_1(X_0(p)_{\mathrm{orb}},\Z)\cong
\begin{cases}
\Z/6\Z\oplus \Z^{2g}, &
\text{if $p\equiv 1$ $\mathrm{mod}$ $12$}, \\
\Z/2\Z\oplus \Z^{2g}, &
\text{if $p\equiv 5$ $\mathrm{mod}$ $12$}, \\
\Z/3\Z\oplus \Z^{2g}, &
\text{if $p\equiv 7$ $\mathrm{mod}$ $12$}, \\
\Z^{2g}, &
\text{if $p\equiv 11$ $\mathrm{mod}$ $12$}.
\end{cases}
\]
Complex conjugation acts on $H_1(X_0(p)_{\mathrm{orb}},\Z)$ by
\[
\big\{0,\frac{1}{a}\big\}\mapsto \big\{0,\frac{1}{p-a}\big\}.
\]
Lemma \ref{lem.algebraic} identifies
$\ocB_2^{[p]}/(\cC\cap \ocB_2^{[p]})$ with the quotient of
$H_1(X_0(p)_{\mathrm{orb}},\Z)$ by the elements of the
form sum of a cycle and its conjugate.

Complex conjugation acts
trivially on $H_1(X_0(p)_{\mathrm{orb}},\Z)_{\mathrm{tors}}$.
When $p\equiv 1$ mod $4$,
intersection number mod $2$ with a conjugation-invariant curve joining the
order $2$ elliptic points splits off
$H_1(X_0(p)_{\mathrm{orb}},\Z)[2]$ equivariantly as a direct summand of
$H_1(X_0(p)_{\mathrm{orb}},\Z)$.
Now $\ocB_2^{[p]}/(\cC\cap \ocB_2^{[p]})$ is a direct sum of
$\Z/2\Z$ when $p\equiv 1$ mod $4$, zero when $p\equiv 3$ mod $4$, and the
quotient of $H_1(X_0(p),\Z)$ by the elements of the
form sum of a cycle and its conjugate.
The latter is accessed by choosing a conjugation-invariant triangulation
of $X_0(p)$ and using spectral sequences relating the equivariant homology
of $X_0(p)$ with the group homology of $H_j(X_0(p),\Z)$,
on the one hand, and the group homology of the groups of $j$-chains on the other,
for $j=0$, $1$, $2$;
cf.\ \cite[\S VII.7]{brown}.
(All group homology is for the group $\Z/2\Z$, corresponding to complex conjugation.)
We omit the details and report only the outcome:
\begin{align*}
H_i\big(\Z/2\Z,H_1(X_0(p),\Z)\big)&=0\qquad\text{for all $i\ge 1$},\\
H^{\Z/2\Z}_j(X_0(p),\Z)&\cong
\begin{cases}
\Z,&\text{if $j=0$}, \\
\Z/2\Z\oplus \Z^g,&\text{if $j=1$}, \\
\Z/2\Z,&\text{if $j\ge 2$}.
\end{cases}
\end{align*}
The vanishing of $H_1\big(\Z/2\Z,H_1(X_0(p),\Z)\big)$ has the consequence that the
subgroup of $H_1(X_0(p),\Z)$ of elements of the form sum of a cycle and
its conjugate is saturated, i.e., has torsion-free quotient.

\appendix
\section{$G$-equivariant Burnside group}
\label{app.BnG}
Let $G$ be a finite abelian group, with character group $A$,
and $X$ a smooth projective variety over an {\em algebraically closed} field of
characteristic zero,  with a faithful action of $G$.
The paper \cite{kontsevichpestuntschinkel} introduced
\begin{itemize}
\item
the abelian group
$$
\cB_n(G),
$$
as the quotient of the $\mathbb Z$-module generated by symbols  $(a_1,\ldots, a_n)$, with $a_i \in A$, and
subject to conditions and
relations similar to those in Definition \ref{def.ocB},
\item
the equivariant Burnside group
$$
\Burn_n(G),
$$
and
\item
the $G$-equivariant birational invariant
$$
\beta(X)\in \Burn_n(G).
$$
\end{itemize}
The invariant $\beta(X)$ is the term
corresponding to $H=G$ in the first sum in the
formula \eqref{eqn.GX} from the Introduction, with
some indices shifted by $1$.
The shift of indices reflects that,
under the assumption that $G$ is nontrivial and
the fixed locus $X^G$ is nonempty,
$X^G$ has positive codimension in $X$.
In this Appendix we explain in detail the formula \eqref{eqn.GX}.

Now, let $k$ be an {\em arbitrary} field of characteristic zero and $X$ a
smooth projective variety over $k$, with
a faithful action of a finite abelian group $G$.
Let $H\subset G$ be a subgroup.
The $H$-fixed locus $X^H$ is smooth,
a finite union of orbits of components.
After removing those where the generic point has stabilizer
strictly larger than $H$ we are left with
$Y_1$, $\dots$, $Y_r$, each with an induced faithful action of the quotient group $G/H$.
Let $Z_1$, $\dots$, $Z_r$ denote the respective quotient varieties.
With the convention from the Introduction, by which we may write
$k(Y)$ when $Y$ is not necessarily irreducible, we have a
$G/H$-Galois algebra extension $k(Y_i)/k(Z_i)$ for every $i$.

\begin{assu}
\label{assu.Hexponent}
For all $H$ and all $k(Y_i)/k(Z_i)$, the field
$k(Z_i)$ contains primitive $e$th roots of unity, where
$e$ is the exponent of $H$, and the homomorphism
\[ H^1(G,k(Y_i)^\times)\to H^1(H,k(Y_i)^\times)^{G/H}=H^1(H,k(Z_i)^\times) \]
of the Hochschild-Serre spectral sequence is surjective.
\end{assu}

The homomorphism in Assumption \ref{assu.Hexponent} is always injective,
since by the Hochschild-Serre spectral sequence
the kernel is $H^1(G/H,k(Y_i)^\times)$, which vanishes by
Hilbert's Theorem 90 (for Galois algebras, this may be
found in \cite[\S 4.3]{JLY}).
Thus, Assumption \ref{assu.Hexponent} implies that it is an isomorphism.

The next result tells us that, that under Assumption \ref{assu.Hexponent},
for every $n$ the non-abelian cohomology set $H^1(G,GL_n(k(Y_i)))$
(see, e.g., \cite[\S 1.3.2]{platonovrapinchuk}) may be identified with
$H^1(H,GL_n(k(Z_i)))$, the set of equivalence classes of
linear $n$-dimensional representations of $H$ over $k(Z_i)$.

\begin{prop}
\label{prop.Hexponent}
Let $G$ be a finite abelian group, $H\subset G$ a subgroup,
$K_0$ a field containing a primitive $e$th root of unity, where
$e$ denotes the exponent of $H$, and
$K$ a $G/H$-Galois algebra over $K_0$.
If the homomorphism $H^1(G,K^\times)\to H^1(H,K_0^\times)$ of
the Hochschild-Serre spectral sequence is surjective,
then for every positive integer $n$
there is a unique bijective map of non-abelian cohomology sets
\[ H^1(G,GL_n(K))\to H^1(H,GL_n(K_0)) \]
that is compatible with
restriction
$$
H^1(G,GL_n(K))\to H^1(H,GL_n(K))
$$ and
extension of scalars
$$
H^1(H,GL_n(K_0))\to H^1(H,GL_n(K)).
$$
\end{prop}

\begin{proof}
We fix a primitive $e$th root of unity $\zeta\in K_0$.
The extension of scalars map from the statement is injective
(by standard representation theory), so
it suffices to exhibit a compatible bijective map of the
indicated non-abelian cohomology sets.
In fact, it suffices to verify the compatibility condition after
replacing $H^1(H,GL_n(K))$ with $H^1(H,GL_n(\widehat{K}))$ for some
\'etale extension $\widehat{K}/K$.

By the structure theorem of
finite abelian groups,
$$
H\cong \Z/n_1\Z\times\dots\times \Z/n_r\Z,
$$
for
$n_1$, $\dots$, $n_r\ge 2$ with
$n_i\mid n_{i+1}$ for $i=1$, $\dots$, $r-1$; we have $e=n_r$.
Let $\zeta\in K_0$ be a primitive $e$th root of unity.
Sending the $i$th generator of $H$ to $\zeta^{e/n_i}$ and all other
generators to $1$, we have an element of $H^1(H,K_0^\times)$ which,
by hypothesis, comes from a $1$-cocycle $(u_{i,g})_{g\in G}$ with
values in $K^\times$.
As remarked above, the homomorphism from the statement is always
injective, therefore $(u_{i,g}^{n_i})_{g\in G}$ is a
$1$-coboundary, i.e., for some $v_i\in K^\times$ we have
$$
u_{i,g}^{n_i}={}^gv_i/v_i\qquad\text{for all}\qquad g\in G.
$$

The data of $(u_{i,g}^{n_i})_{i,g}$ and $(v_i)_i$ give us a
way to assign, functorially, a $H$-Galois algebra over an
\'etale $K_0$-algebra $L_0$ to every
$G$-Galois $L/L_0$ with $G$-equivariant $K_0$-algebra homomorphism $K\to L$.
Specifically, given $G$-Galois $L/L_0$ with $\iota\colon K\to L$ we
apply Hilbert's Theorem 90 to obtain $w_i\in L^\times$ for every $i$,
satisfying $\iota(u_{i,g})={}^gw_i/w_i$ for all $g$.
Now $\iota(v_i)w_i^{-n_i}$ is Galois-invariant, i.e., lies in $L_0$,
and is unique up to multiplication by an element of $(L_0^\times)^{n_i}$,
for every $i$
(since $w_i$ is unique up to multiplication by an element of $L_0$);
we associate the $H$-Galois algebra
\[
L_0[t_1,\dots,t_r]/(t_1^{n_1}-\iota(v_1)w_1^{-n_1},\dots,
t_r^{n_r}-\iota(v_r)w_r^{-n_r}).
\]
The functorial association is fully faithful.
We deduce that it is essentially surjective,
hence gives an equivalence of categories, using that any
$H$-Galois $L'_0/L_0$ is trivialized by an \'etale extension of $L_0$
(e.g., $L'_0\otimes_{L_0}L'_0\cong \prod_{h\in H}L'_0$) and described
up to isomorphism by an $H$-valued $1$-cocycle (in Galois cohomology).

Finally, the non-abelian cohomology sets from the statement are
invariants of the categories described in the previous paragraph.
A unique element of $H^1(G,GL_n(K))$, respectively,
$H^1(H,GL_n(K_0))$, is associated to a
functorial association of a free $L_0$-module of rank $n$ to every
$G$-Galois $L/L_0$ with $G$-equivariant $K_0$-algebra homomorphism $K\to L$,
respectively, to every $H$-Galois algebra over $L_0$.
The equivalence of categories of the previous paragraph identifies these
two non-abelian cohomology sets.
For the compatibility, we use the category of
$H$-Galois $L'_0/L_0$ with extension of the $K_0$-algebra structure of $L_0$
to a $K$-algebra structure, with functor
\[
\textstyle
(L'_0/L_0,K\stackrel{\beta}\to L_0)\mapsto
\big((\prod_{g\in G}L'_0)^H/L_0,x\mapsto (\beta({}^gx))_{g\in G}\big)
\]
whose composite with the functor of the previous paragraph becomes
naturally isomorphic to the forgetful functor
after replacing $K$ by a suitable \'etale extension $\widehat{K}$.
Functorial associations on this category, as above,
are characterized by elements of $H^1(H,GL_n(\widehat{K}))$,
such that the functors give rise to the relevant maps on
non-abelian cohomology sets.
We obtain the required compatibility.
\end{proof}

Assumption \ref{assu.Hexponent} always holds when
$k$ contains all roots of unity.
In general,
after performing the divisorialification procedure (see Section \ref{sec.birat})
the following stronger condition will hold.

\begin{assu}
\label{assu.strong}
For all $H$ and all $k(Y_i)/k(Z_i)$ the field
$k(Z_i)$ contains primitive $e$th roots of unity, where
$e$ is the exponent of $H$, and the composition of the
map from Assumption \ref{assu.Hexponent} with the
restriction
\[ \Pic^G(X)\to \Pic^G(\Spec(k(Y_i)))\cong H^1(G,k(Y_i)^\times) \]
is a surjective homomorphism
\[ \Pic^G(X)\to H^1(H,k(Z_i)^\times). \]
\end{assu}

\begin{prop}
\label{prop.divisorialification}
Let $k$ be a field of characteristic zero and $X$ a smooth projective variety
over $k$ with a faithful action of a finite abelian group $G$.
Then there exist smooth projective varieties with $G$-action
and $G$-equivariant morphisms
\[ X'=X_n\to \dots\to X_1\to X_0=X, \]
each the blow-up of a smooth $G$-invariant subscheme,
such that $X' \actsfromright G$ satisfies Assumption \ref{assu.strong}.
\end{prop}

\begin{proof}
The divisorialification procedure, applied to $[X/G]$, translates into
a sequence of blow-ups that meets the stated conditions.
\end{proof}

\begin{prop}
\label{prop.equivweakfact}
Let $k$ be a field of characteristic zero,
$X$ and $X'$ smooth projective varieties with faithful actions of a
finite abelian group $G$ satisfying Assumption \ref{assu.strong}, and
$$
\varphi\colon X'\dashrightarrow X
$$
a $G$-equivariant birational map
restricting to an isomorphism over open $U\subset X$.
Then there exists a weak factorization of $\varphi$, where each
map is, or is inverse to, the blow-up of a smooth $G$-invariant subscheme
disjoint from $U$
and the intermediate projective varieties
with $G$-action satisfy Assumption \ref{assu.strong}.
\end{prop}

\begin{proof}
There is no loss of generality in supposing $\varphi$ to be a morphism.
By Assumption \ref{assu.strong} there is a finite collection of
$G$-linearized line bundles on $X$ whose classes in $\Pic^G(X)$ map
to a generating system of
$H^1(H,k(Z_i)^\times)$ for every $H$ and $k(Y_i)/k(Z_i)$.
We represent these by a $\G_m^r$-torsor $V\to X$ (where $r$ denotes the
number of line bundles), with
$G$-action on $V$ determined by the linearizations.
By \cite[Rmk.\ 7.14]{bergh}, Assumption \ref{assu.strong} for $X$ implies
that the $G$-action on $V$ is free.
Let $V'=X'\times_XV$.
Functorial weak factorization \cite{abramovichtemkin} provides
compatible weak factorizations of $X'\to X$
($G$-invariant, maintaining isomorphisms over $U$) and $V'\to V$, with
$\G_m^r$-torsor structure preserved throughout the weak factorization.
Since, for any subgroup $H\subset G$, the property of having a point with
stabilizer exactly $H$ is preserved under blow-up of a smooth $G$-invariant
subscheme, we see that the freeness of the $G$-action is maintained throughout
the weak factorization.
By \cite[Rmk.\ 7.14]{bergh}, again, we
deduce that Assumption \ref{assu.strong} is maintained throughout the
weak factorization.
\end{proof}

As stated in the Introduction,
$\Burn_n(G)$ is a quotient of the free abelian group on triples
consisting of a subgroup $H$ of $G$,
a $G/H$-Galois algebra extension $K$ of a field $K_0$ of
some transcendence degree $d\le n$ over $k$ satisfying
Assumption \ref{assu.Hexponent}, and a faithful
$(n-d)$-dimensional linear representation of $H$ over $K_0$ with
trivial space of invariants.
Any such representation splits as a sum of one-dimensional representations,
so we may write the representation as a sequence of $(n-d)$ nonzero elements
that generate the Cartier dual $\widehat{H}$.
Triples related by an equivariant isomorphism of algebras over $k$ are
regarded as equivalent, as are those which differ by a permutation of the
elements of $\widehat{H}$.
Then we identify
$$
[G/H \actsfromleft K,(a_1,\dots,a_{n-d})]
$$
with,
for any $2\le j\le n-d$,
the sum over pairs $(I,C_I)$ with
$\emptyset\ne I\subset \{1,\dots,j\}$ and $C_I$ a coset (of some subgroup) in $A$,
satisfying
\[ I=\{1\le i\le j\,|\,a_i\in C_I\},\qquad
C_I=a_{i_0}+\langle a_i-a_{i_0}\rangle_{i\in I}\,\,(i_0\in I), \]
of the generator indexed by the triple that we get as follows:
\begin{itemize}
\item Set $H_I=\bigcap_{\chi\in A_I}\ker(\chi)$, where
$A_I=\langle a_i-a_{i_0}\rangle_{i\in I}$.
So,
$$
\widehat{H_I}=A/A_I.
$$
\item Let a representation of $H_I$ be determined by $a_{i_0}$ together with
$a_i-a_{i_0}$ for all $i\le j$ with $i\notin I$
and $a_{j+1}$, $\dots$, $a_{n-d}$;
this gives a sequence of elements of $A/A_I$ that is
independent of $i_0\in I$.
\item Writing $I=\{i_0,i_1,\dots,i_{|I|-1}\}$ and letting
$$
b^{(i)}=(b^{(i)}_g)_{g\in G}
$$
denote the $1$-cocycle with values in $K^\times$
corresponding by Assumption \ref{assu.Hexponent}
to $a_i\in \widehat{H}=H^1(H,K_0^\times)$, we let $G/H_I$ act on the field
$K(t_1,\dots,t_{|I|-1})$ by the given action on $K$ and by
$b^{(i_c)}/b^{(i_0)}$ on $t_c$, for $c=1$, $\dots$, $|I|-1$.
\end{itemize}
It may happen, for some $(I,C_I)$, that
$0$ appears in the sequence of elements of $A/A_I$.
Those $(I,C_I)$ are simply omitted from the sum, leaving a sum of generators
associated with valid triples.

\bibliographystyle{plain}
\bibliography{burnside}

\begin{thebibliography}{10}

\bibitem{AGV}
D.~Abramovich, T.~Graber, and A.~Vistoli.
\newblock Gromov-{W}itten theory of {D}eligne-{M}umford stacks.
\newblock {\em Amer. J. Math.}, 130(5):1337--1398, 2008.

\bibitem{AKMW}
D.~Abramovich, K.~Karu, K.~Matsuki, and J.~W{\l}odarczyk.
\newblock Torification and factorization of birational maps.
\newblock {\em J. Amer. Math. Soc.}, 15(3):531--572, 2002.

\bibitem{abramovichtemkin}
D.~Abramovich and M.~Temkin.
\newblock Functorial factorization of birational maps for qe schemes in
  characteristic 0.
\newblock {\em Algebra Number Theory}, 13(2):379--424, 2019.

\bibitem{behrendnoohi}
K.~Behrend and B.~Noohi.
\newblock Uniformization of {D}eligne-{M}umford curves.
\newblock {\em J. Reine Angew. Math.}, 599:111--153, 2006.

\bibitem{bergh}
D.~Bergh.
\newblock Functorial destackification of tame stacks with abelian stabilisers.
\newblock {\em Compos. Math.}, 153(6):1257--1315, 2017.

\bibitem{berghrydh}
D.~Bergh and D.~Rydh.
\newblock Functorial destackification and weak factorization of orbifolds,
  2019.
\newblock {\tt arXiv:1905.00872}.

\bibitem{bierstonemilman}
E.~Bierstone and P.~D. Milman.
\newblock Canonical desingularization in characteristic zero by blowing up the
  maximum strata of a local invariant.
\newblock {\em Invent. Math.}, 128(2):207--302, 1997.

\bibitem{bittner}
F.~Bittner.
\newblock The universal {E}uler characteristic for varieties of characteristic
  zero.
\newblock {\em Compos. Math.}, 140(4):1011--1032, 2004.

\bibitem{brown}
K.~S. Brown.
\newblock {\em Cohomology of groups}, volume~87 of {\em Graduate Texts in
  Mathematics}.
\newblock Springer-Verlag, New York-Berlin, 1982.

\bibitem{cadman}
C.~Cadman.
\newblock Using stacks to impose tangency conditions on curves.
\newblock {\em Amer. J. Math.}, 129(2):405--427, 2007.

\bibitem{JLY}
C.~U. Jensen, A.~Ledet, and N.~Yui.
\newblock {\em Generic polynomials: Constructive aspects of the inverse
  {G}alois problem}.
\newblock Cambridge University Press, Cambridge, 2002.

\bibitem{keelmori}
S.~Keel and S.~Mori.
\newblock Quotients by groupoids.
\newblock {\em Ann. of Math. (2)}, 145(1):193--213, 1997.

\bibitem{kleimanaltman}
S.~L. Kleiman and A.~B. Altman.
\newblock Bertini theorems for hypersurface sections containing a subscheme.
\newblock {\em Comm. Algebra}, 7(8):775--790, 1979.

\bibitem{kontsevichpestuntschinkel}
M.~Kontsevich, V.~Pestun, and Yu. Tschinkel.
\newblock Equivariant birational geometry and modular symbols, 2019.
\newblock {\tt arXiv:1902.09894}.

\bibitem{kontsevichtschinkel}
M.~Kontsevich and Yu. Tschinkel.
\newblock Specialization of birational types.
\newblock {\em Invent. Math.}, 217(2):415--432, 2019.

\bibitem{kreschseattle}
A.~Kresch.
\newblock On the geometry of {D}eligne-{M}umford stacks.
\newblock In {\em Algebraic geometry---{S}eattle 2005. {P}art 1}, volume~80 of
  {\em Proc. Sympos. Pure Math.}, pages 259--271. Amer. Math. Soc., Providence,
  RI, 2009.

\bibitem{kreschdestackification}
A.~Kresch.
\newblock Destackification with restricted root operations.
\newblock {\em Eur. J. Math.}, 4(4):1421--1432, 2018.

\bibitem{maninparabolic}
Yu.~I. Manin.
\newblock Parabolic points and zeta functions of modular curves.
\newblock {\em Izv. Akad. Nauk SSSR Ser. Mat.}, 36:19--66, 1972.

\bibitem{moerdijk}
I.~Moerdijk.
\newblock Orbifolds as groupoids: an introduction.
\newblock In {\em Orbifolds in mathematics and physics ({M}adison, {WI},
  2001)}, volume 310 of {\em Contemp. Math.}, pages 205--222. Amer. Math. Soc.,
  Providence, RI, 2002.

\bibitem{oesinghaus}
J.~Oesinghaus.
\newblock Conic bundles and iterated root stacks.
\newblock {\em Eur. J. Math.}, 5(2):518--527, 2019.

\bibitem{platonovrapinchuk}
V.~P. Platonov and A.~S. Rapinchuk.
\newblock {\em Algebraicheskie gruppy i teoriya chisel}.
\newblock Nauka, Moscow, 1991.

\bibitem{villamayor}
O.~E. Villamayor~U.
\newblock Patching local uniformizations.
\newblock {\em Ann. Sci. \'{E}cole Norm. Sup. (4)}, 25(6):629--677, 1992.

\end{thebibliography}

\end{document}